\documentclass[12pt]{amsart}
\usepackage{amssymb, amsfonts, amsmath, enumerate, extarrows}
\usepackage{tikz}
\usepackage{subfigure}
\usepackage{color}
\usepackage[all]{xy}
\usepackage[enableskew,vcentermath]{youngtab}

\usetikzlibrary{arrows,shapes,automata,snakes,backgrounds,petri,positioning}

\theoremstyle{plain}
\newtheorem{thm}{Theorem}[section]
\newtheorem{lem}[thm]{Lemma}

\newtheorem{cor}[thm]{Corollary}

\newtheorem{prop}[thm]{Proposition}

\newtheorem{ques}{Question}

\theoremstyle{definition}

\newtheorem{example}[thm]{Example}

\newcommand{\bm}[1]{\mbox{\boldmath $#1$}}

\DeclareMathOperator{\sh}{sh}
\DeclareMathOperator{\h}{ht}

\newcommand{\sch}{\mbox {\textsf {Sch}}}

\title[Involution Products]{How to write a permutation as a product of involutions (and why you might care)}

\author[T. K. Petersen]{T. Kyle Petersen$^*$}
\thanks{$^*$ Research partially supported by an NSA Young Investigator grant}
\author[B. E. Tenner]{Bridget Eileen Tenner$^\dagger$}
\thanks{$^\dagger$ Research partially supported by a DePaul University Faculty Summer Research Grant.} 

\address{Department of Mathematical Sciences, DePaul University, Chicago, IL 60614, USA}
\email{tkpeters@math.depaul.edu\\bridget@math.depaul.edu}

\begin{document}

\begin{abstract}
It is well-known that any permutation can be written as a product of two involutions. We provide an explicit formula for the number of ways to do so, depending only on the cycle type of the permutation.

In many cases, these numbers are sums of absolute values of irreducible characters of the symmetric group evaluated at the same permutation, although apart from the case where all cycles are the same size, we have no good explanation for why this should be so.  
\end{abstract}

\maketitle

\section{Introduction}

The authors were interested in finding a combinatorial model counted by the following sequence of integers: \begin{equation}\label{seq:A} 1, 4, 9, 27, 61, 185, 469, \ldots, \end{equation} And upon putting these terms into The On-Line Encyclopedia of Integer Sequences, we received a match in sequence A164342 \cite{OEIS}. That entry was a stub however, simply defined as the row sums of the following table, given in entry A164341:
\begin{equation}\label{arr:A} \begin{array}{c | c c c c c c c c c c c c c c  c}
n \backslash m & 1 & 2 & 3 & \cdots \\
\hline
1 & 1\\
2 & 2 & 2\\
3 & 3 & 2 & 4\\
4 & 4 & 3 & 6 & 4 & 10\\
5 & 5 & 4 & 6 & 6 & 6 & 8 & 26\\
6 & 6 & 5 & 8 & 12 & 8 & 6 & 20 & 12 & 12 & 20 & 76\\
7 & 7 & 6 & 10 & 12 & 10 & 8 & 12 & 18 & 16 & 12 & 20 & 30 & 24 & 52 & 232
\end{array}\end{equation}
But this was fantastic news, for our sequence could be refined in exactly the same fashion!

We had generated the terms in \eqref{seq:A} by computing the sums of the absolute values of the irreducible characters of the symmetric group $S_n$, more precisely:
$$\sum_{\lambda, \mu \vdash n} |\chi^\lambda_\mu|,$$
and the entries in array \eqref{arr:A} by fixing $\mu$ and considering
$$\sum_{\lambda \vdash n} |\chi^\lambda_\mu|.$$
On the other hand, A164341 said entry $(n,m)$ ``counts the decompositions into involutions of a permutation that has a cycle structure given by the $m$th partition of $n$." 

Could this mean there might be a way to compute irreducible characters of $S_n$ by counting involution products in the right way? Well, no. Our refinement of the sequence in \eqref{seq:A} and the array in A164341 agreed for $n\le 7$, but diverged at $n=8$.  Here is the row for $n=8$ in our table, sitting atop the row for $n=8$ in A164341.
$$\tiny \begin{array}{c | c c c ccccccccccccccccccc}
{\displaystyle \sum_{\ \lambda \vdash 8_{\ }} |\chi^\lambda_\mu| } & 8 &7 &12 &15 &20 &12 &10 &12 &\bm{20}& 24 &20& \bm{12}& 24 &18 &76 &\bm{36}& 24 &\bm{36} &78 &\bm{52} &\bm{148} &764 \\ \hline \\
\mbox{A164341} & 8 &7 &12 &15 &20 &12 &10 &12 &\bm{24}& 24 &20& \bm{16}& 24 &18 &76 &\bm{40}& 24 &\bm{40} &78 &\bm{60} &\bm{152} &764 \end{array}$$
(The differences are highlighted in boldface.) And here are the next few terms in sequence \eqref{seq:A}, with the next few terms of A164342 below.
$$\begin{array}{c | c c c c c c}
 n & 8 & 9 & 10 & 11 & 12 & 13 \\ \hline \hline \\
   \sum_{\lambda,\mu \vdash n} |\chi^\lambda_\mu| & 1428 & 4292 & 14456 & 50040 & 186525 & 724023 \\
   \\
   \mbox{A164342} &  1456 & 4368 & 14720 & 50800 & 190149 & 735451 \\ \\\hline
   \mbox{(overcount)} & 28 & 76 & 264 & 760 & 3624 & 11428
   \end{array}$$

In this note, we study the ways in which a permutation can be expressed as a product of two involutions. This puts entries A164341 and A164342 of \cite{OEIS} on solid footing, as there seem to be no references to this question in the literature. Because the motivation for this problem is the study of irreducible characters of the symmetric group, we try to explain what we can of the connection between the two topics, using ideas of White \cite{white} and Stanton and White \cite{stanton-white}. 

Section \ref{sec:prod} investigates involution products from a purely combinatorial point of view. Section \ref{sec:connections} discusses the relation to irreducible characters.

\section{Products of involutions in $S_n$}\label{sec:prod}


For a permutation $\sigma \in S_n$, we typically use cycle notation in this work.  For example, $\sigma = (1236547)$ is the map $1 \mapsto 2 \mapsto 3 \mapsto 6 \mapsto 5 \mapsto 4 \mapsto 7 \mapsto 1$, while $\tau = (135)(26)(4)(7)$ maps $1 \mapsto 3 \mapsto 5 \mapsto 1$, $2 \mapsto 6 \mapsto 2$, and fixes $4$ and $7$.
It is useful to draw our permutations as digraphs on the set $[n]:=\{1,2,\ldots,n\}$, in order to visualize the cycle structure. For example,
$$\tau = \hspace{.1in}
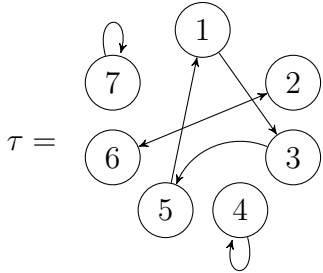
\begin{minipage}{1.5in}
\begin{tikzpicture}[node distance=.25cm,>=stealth',bend angle=45,auto]
\tikzstyle{state}=[circle,draw=black,minimum size=4mm]
\node[state] (a) {$1$};
\node[state] (b) [below right=of a,xshift=5mm] {$2$};
\node[state] (c) [below=of b] {$3$};
\node[state] (d) [below left=of c] {$4$};
\node[state] (e) [left=of d] {$5$};
\node[state] (f) [above left=of e] {$6$};
\node[state] (g) [above=of f] {$7$};
\path[->] (a) edge (c);
\path[->] (c) edge[bend right] (e);
\path[->] (e) edge (a);
\path[<->] (b) edge (f);
\path[->] (d) edge [loop below] (d);
\path[->] (g) edge [loop above] (g);
\end{tikzpicture}
\end{minipage}$$
The cycle structure of a permutation is encoded by a \emph{partition}. A partition $\lambda \vdash n$ is a collection of positive integers whose sum is $n$. We usually order the parts of $\lambda = (\lambda_1,\lambda_2,\ldots)$ in nonincreasing order; that is, $\lambda_1 \geq \lambda_2 \geq \cdots$ and $\sum \lambda_i = n$. An alternate notation is to let $j_i$ denote the number of parts of size $i$ in $\lambda$, and to write $\lambda = 1^{j_1} 2^{j_2} \cdots$, often suppressing any $j_i = 1$. For example, $\tau$ above has two one-cycles, one two-cycle, and one three-cycle; we encode this information with the partition $\lambda =(3,2,1,1)=1^22^13^1 = 1^223$.

The product of two permutations can be found by superimposing two digraphs of the form described above. For example, if we wish to compute the product of $\sigma$ and $\tau$ from above, we draw the edges of $\sigma$ and $\tau$ with different colors, say red and blue, respectively. Then $\sigma\tau(i) = j$ if there is a directed path $i {\color{blue} \rightarrow } \tau(i) {\color{red} \rightarrow} j$. For example,
$$\sigma\tau = \hspace{.1in}
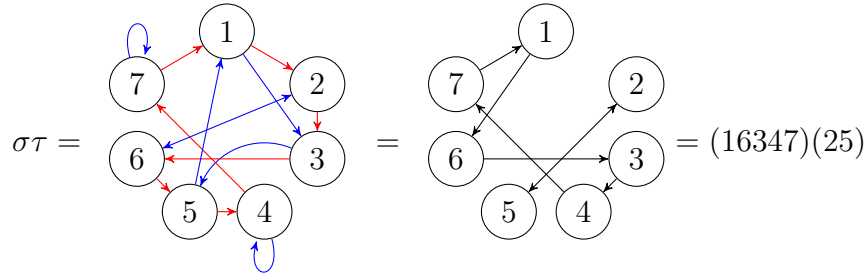
\begin{minipage}{1.25in}
\begin{tikzpicture}[node distance=.25cm,>=stealth',bend angle=45,auto]
\tikzstyle{state}=[circle,draw=black,minimum size=4mm]
\node[state] (a) {$1$};
\node[state] (b) [below right=of a,xshift=5mm] {$2$};
\node[state] (c) [below=of b] {$3$};
\node[state] (d) [below left=of c] {$4$};
\node[state] (e) [left=of d] {$5$};
\node[state] (f) [above left=of e] {$6$};
\node[state] (g) [above=of f] {$7$};
\path[->,red] (a) edge (b);
\path[->,blue] (a) edge (c);
\path[->,red] (b) edge (c);
\path[->,blue] (c) edge[bend right] (e);
\path[->,red] (c) edge (f);
\path[->,blue] (d) edge [loop below] (d);
\path[->,red] (d) edge (g);
\path[->,blue] (e) edge (a);
\path[->,red] (e) edge (d);
\path[->,red] (f) edge (e);
\path[->,red] (g) edge (a);
\path[->,blue] (g) edge [loop above] (g);
\path[<->,blue] (b) edge (f);
\end{tikzpicture}
\end{minipage}
\hspace{.1in} = \hspace{.1in}
\begin{minipage}{1.25in}
\begin{tikzpicture}[node distance=.25cm,>=stealth',bend angle=45,auto]
\tikzstyle{state}=[circle,draw=black,minimum size=4mm]
\node[state] (a) {$1$};
\node[state] (b) [below right=of a,xshift=5mm] {$2$};
\node[state] (c) [below=of b] {$3$};
\node[state] (d) [below left=of c] {$4$};
\node[state] (e) [left=of d] {$5$};
\node[state] (f) [above left=of e] {$6$};
\node[state] (g) [above=of f] {$7$};
\path[->] (a) edge (f);
\path[->] (c) edge (d);
\path[->] (d) edge (g);
\path[->] (f) edge (c);
\path[->] (g) edge (a);
\path[<->] (b) edge (e);
\path[->,white] (d) edge [loop below] (d); 
\end{tikzpicture}
\end{minipage}
 = (16347)(25)$$

We now consider the special case of \emph{involutions}; that is, permutations whose squares are the identity. In an involution, each cycle must have size one or two, meaning that we can represent such elements graphically by \emph{partial matchings}. For example, 
$$\iota = (14)(23)(57)(6) = \hspace{.1in}
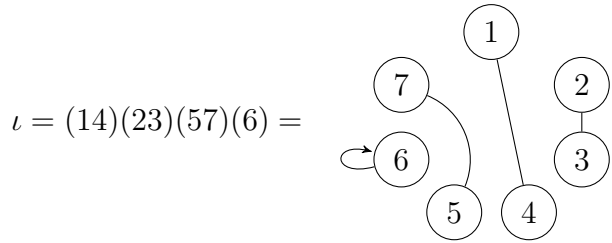
\begin{minipage}{1.5in}
\begin{tikzpicture}[node distance=.25cm,>=stealth',bend angle=45,auto]
\tikzstyle{state}=[circle,draw=black,minimum size=4mm]
\node[state] (a) {$1$};
\node[state] (b) [below right=of a,xshift=5mm] {$2$};
\node[state] (c) [below=of b] {$3$};
\node[state] (d) [below left=of c] {$4$};
\node[state] (e) [left=of d] {$5$};
\node[state] (f) [above left=of e] {$6$};
\node[state] (g) [above=of f] {$7$};
\path (a) edge (d);
\path (b) edge (c);
\path (e) edge[bend right] (g);
\path (f) edge[loop left] (f);
\end{tikzpicture}
\end{minipage}$$

The product of two involutions in $S_n$ can then be expressed as a two-colored graph on $[n]$ (with loops) for which every vertex has precisely one edge of each color, counting a loop as a single edge. We call such a graph an \emph{involution product graph}. For example, with the involution $\iota$ as above, and $\kappa = (1)(27)(35)(46)$, we have
$$\iota\kappa = \hspace{.1in}
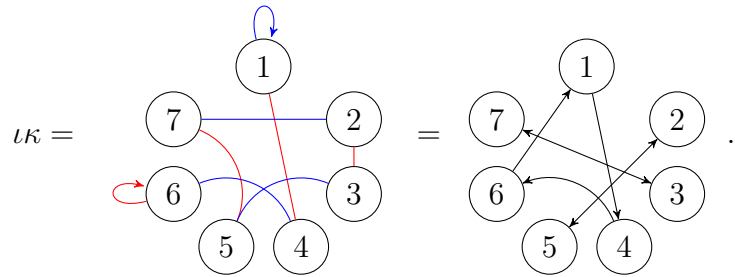
\begin{minipage}{1.5in}
\begin{tikzpicture}[node distance=.25cm,>=stealth',bend angle=45,auto]
\tikzstyle{state}=[circle,draw=black,minimum size=4mm]
\node[state] (a) {$1$};
\node[state] (b) [below right=of a,xshift=5mm] {$2$};
\node[state] (c) [below=of b] {$3$};
\node[state] (d) [below left=of c] {$4$};
\node[state] (e) [left=of d] {$5$};
\node[state] (f) [above left=of e] {$6$};
\node[state] (g) [above=of f] {$7$};
\path[red] (a) edge (d);
\path[red] (b) edge (c);
\path[red] (e) edge[bend right] (g);
\path[red] (f) edge [loop left] (f);
\path[blue] (a) edge [loop above] (a);
\path[blue] (b) edge (g);
\path[blue] (c) edge[bend right] (e);
\path[blue] (d) edge[bend right] (f);
\end{tikzpicture}
\end{minipage}
\hspace{.1in} = \hspace{.1in}
\begin{minipage}{1.25in}
\begin{tikzpicture}[node distance=.25cm,>=stealth',bend angle=45,auto]
\tikzstyle{state}=[circle,draw=black,minimum size=4mm]
\node[state] (a) {$1$};
\node[state] (b) [below right=of a,xshift=5mm] {$2$};
\node[state] (c) [below=of b] {$3$};
\node[state] (d) [below left=of c] {$4$};
\node[state] (e) [left=of d] {$5$};
\node[state] (f) [above left=of e] {$6$};
\node[state] (g) [above=of f] {$7$};
\path[->] (a) edge (d);
\path[<->] (c) edge (g);
\path[->] (d) edge[bend right] (f);
\path[->] (f) edge (a);
\path[<->] (b) edge (e);
\path[->,white] (a) edge [loop above] (a); 
\end{tikzpicture}
\end{minipage}\hspace{.1in}.$$

We see from this example that the set of involutions is not closed under multiplication. Indeed, it is well-known that any permutation can be written as a product of involutions (see \cite[Exercise 10.1.17]{scott}), often in many different ways. Our goal is to describe, and to count, all the ways in which this can be done.

If an involution product graph defines a permutation $\rho$, we say it is an \emph{involution graph for $\rho$}. In what follows, we identify the set of pairs of involutions whose product is $\rho$ with the set of involution graphs for $\rho$. Let $N(\rho)$ denote the number of ways $\rho \in S_n$ can be written as a product of involutions. To be clear, \[ N(\rho) = |\{ (\sigma,\tau) \in S_n^2 : \sigma^2 = \tau^2 =1, \sigma\tau = \rho\}| = |\{ \mbox{involution graphs for } \rho \}|.\]

\begin{example}
$N((123)(456)) = 12$, as depicted in Figure~\ref{fig:(123)(456) decompositions}.
\end{example}

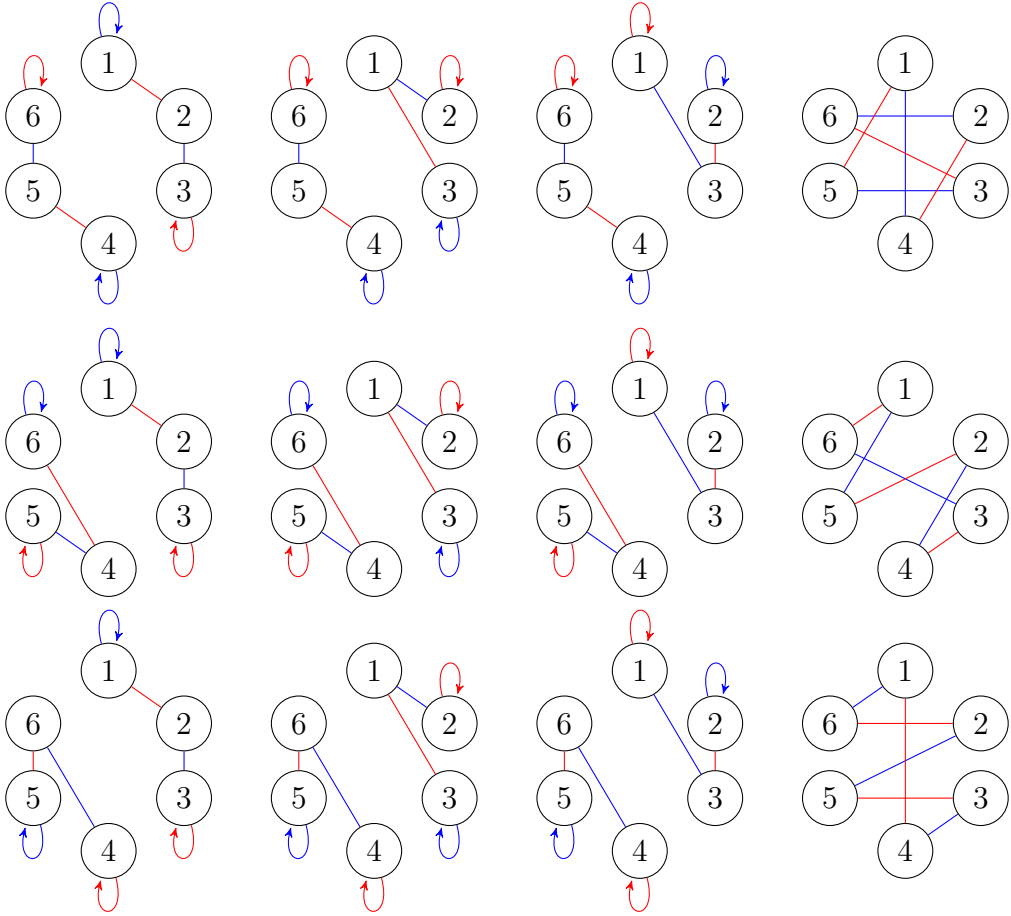
\begin{figure}[htbp]
\begin{center}
\begin{tikzpicture}[node distance=.25cm,>=stealth',bend angle=45,auto]
\tikzstyle{state}=[circle,draw=black,minimum size=4mm]
\node[state] (a) {$1$};
\node[state] (b) [below right=of a,xshift=3mm] {$2$};
\node[state] (c) [below=of b] {$3$};
\node[state] (d) [below left=of c,xshift=-3mm] {$4$};
\node[state] (e) [above left=of d,xshift=-3mm] {$5$};
\node[state] (f) [above=of e] {$6$};
\path[red] (a) edge (b);
\path[red] (d) edge (e);
\path[red] (c) edge [loop below] (c);
\path[red] (f) edge [loop above] (f);
\path[blue] (b) edge (c);
\path[blue] (e) edge (f);
\path[blue] (a) edge [loop above] (a);
\path[blue] (d) edge [loop below] (d);
\end{tikzpicture}
\hspace{.2in}
\begin{tikzpicture}[node distance=.25cm,>=stealth',bend angle=45,auto]
\tikzstyle{state}=[circle,draw=black,minimum size=4mm]
\node[state] (a) {$1$};
\node[state] (b) [below right=of a,xshift=3mm] {$2$};
\node[state] (c) [below=of b] {$3$};
\node[state] (d) [below left=of c,xshift=-3mm] {$4$};
\node[state] (e) [above left=of d,xshift=-3mm] {$5$};
\node[state] (f) [above=of e] {$6$};
\path[red] (a) edge (c);
\path[red] (d) edge (e);
\path[red] (b) edge [loop above] (b);
\path[red] (f) edge [loop above] (f);
\path[blue] (b) edge (a);
\path[blue] (e) edge (f);
\path[blue] (c) edge [loop below] (c);
\path[blue] (d) edge [loop below] (d);
\end{tikzpicture}
\hspace{.2in}
\begin{tikzpicture}[node distance=.25cm,>=stealth',bend angle=45,auto]
\tikzstyle{state}=[circle,draw=black,minimum size=4mm]
\node[state] (a) {$1$};
\node[state] (b) [below right=of a,xshift=3mm] {$2$};
\node[state] (c) [below=of b] {$3$};
\node[state] (d) [below left=of c,xshift=-3mm] {$4$};
\node[state] (e) [above left=of d,xshift=-3mm] {$5$};
\node[state] (f) [above=of e] {$6$};
\path[red] (b) edge (c);
\path[red] (d) edge (e);
\path[red] (a) edge [loop above] (a);
\path[red] (f) edge [loop above] (f);
\path[blue] (c) edge (a);
\path[blue] (e) edge (f);
\path[blue] (b) edge [loop above] (b);
\path[blue] (d) edge [loop below] (d);
\end{tikzpicture}
\hspace{.2in}
\begin{tikzpicture}[node distance=.25cm,>=stealth',bend angle=45,auto]
\tikzstyle{state}=[circle,draw=black,minimum size=4mm]
\node[state] (a) {$1$};
\node[state] (b) [below right=of a,xshift=3mm] {$2$};
\node[state] (c) [below=of b] {$3$};
\node[state] (d) [below left=of c,xshift=-3mm] {$4$};
\node[state] (e) [above left=of d,xshift=-3mm] {$5$};
\node[state] (f) [above=of e] {$6$};
\path[blue] (a) edge (d);
\path[blue] (b) edge (f);
\path[blue] (c) edge (e);
\path[red] (b) edge (d);
\path[red] (c) edge (f);
\path[red] (a) edge (e);
\path[white] (a) edge [loop above] (a); 
\path[white] (d) edge [loop below] (d); 
\end{tikzpicture}\\
\begin{tikzpicture}[node distance=.25cm,>=stealth',bend angle=45,auto]
\tikzstyle{state}=[circle,draw=black,minimum size=4mm]
\node[state] (a) {$1$};
\node[state] (b) [below right=of a,xshift=3mm] {$2$};
\node[state] (c) [below=of b] {$3$};
\node[state] (d) [below left=of c,xshift=-3mm] {$4$};
\node[state] (e) [above left=of d,xshift=-3mm] {$5$};
\node[state] (f) [above=of e] {$6$};
\path[red] (a) edge (b);
\path[red] (d) edge (f);
\path[red] (c) edge [loop below] (c);
\path[red] (e) edge [loop below] (e);
\path[blue] (b) edge (c);
\path[blue] (e) edge (d);
\path[blue] (a) edge [loop above] (a);
\path[blue] (f) edge [loop above] (f);
\end{tikzpicture}
\hspace{.2in}
\begin{tikzpicture}[node distance=.25cm,>=stealth',bend angle=45,auto]
\tikzstyle{state}=[circle,draw=black,minimum size=4mm]
\node[state] (a) {$1$};
\node[state] (b) [below right=of a,xshift=3mm] {$2$};
\node[state] (c) [below=of b] {$3$};
\node[state] (d) [below left=of c,xshift=-3mm] {$4$};
\node[state] (e) [above left=of d,xshift=-3mm] {$5$};
\node[state] (f) [above=of e] {$6$};
\path[red] (a) edge (c);
\path[red] (d) edge (f);
\path[red] (b) edge [loop above] (b);
\path[red] (e) edge [loop below] (e);
\path[blue] (b) edge (a);
\path[blue] (e) edge (d);
\path[blue] (c) edge [loop below] (c);
\path[blue] (f) edge [loop above] (f);
\end{tikzpicture}
\hspace{.2in}
\begin{tikzpicture}[node distance=.25cm,>=stealth',bend angle=45,auto]
\tikzstyle{state}=[circle,draw=black,minimum size=4mm]
\node[state] (a) {$1$};
\node[state] (b) [below right=of a,xshift=3mm] {$2$};
\node[state] (c) [below=of b] {$3$};
\node[state] (d) [below left=of c,xshift=-3mm] {$4$};
\node[state] (e) [above left=of d,xshift=-3mm] {$5$};
\node[state] (f) [above=of e] {$6$};
\path[red] (b) edge (c);
\path[red] (d) edge (f);
\path[red] (a) edge [loop above] (a);
\path[red] (e) edge [loop below] (e);
\path[blue] (c) edge (a);
\path[blue] (e) edge (d);
\path[blue] (b) edge [loop above] (b);
\path[blue] (f) edge [loop above] (f);
\end{tikzpicture}
\hspace{.2in}
\begin{tikzpicture}[node distance=.25cm,>=stealth',bend angle=45,auto]
\tikzstyle{state}=[circle,draw=black,minimum size=4mm]
\node[state] (a) {$1$};
\node[state] (b) [below right=of a,xshift=3mm] {$2$};
\node[state] (c) [below=of b] {$3$};
\node[state] (d) [below left=of c,xshift=-3mm] {$4$};
\node[state] (e) [above left=of d,xshift=-3mm] {$5$};
\node[state] (f) [above=of e] {$6$};
\path[blue] (a) edge (e);
\path[blue] (b) edge (d);
\path[blue] (c) edge (f);
\path[red] (b) edge (e);
\path[red] (c) edge (d);
\path[red] (a) edge (f);
\path[white] (a) edge [loop above] (a); 
\end{tikzpicture}\\
\begin{tikzpicture}[node distance=.25cm,>=stealth',bend angle=45,auto]
\tikzstyle{state}=[circle,draw=black,minimum size=4mm]
\node[state] (a) {$1$};
\node[state] (b) [below right=of a,xshift=3mm] {$2$};
\node[state] (c) [below=of b] {$3$};
\node[state] (d) [below left=of c,xshift=-3mm] {$4$};
\node[state] (e) [above left=of d,xshift=-3mm] {$5$};
\node[state] (f) [above=of e] {$6$};
\path[red] (a) edge (b);
\path[red] (f) edge (e);
\path[red] (c) edge [loop below] (c);
\path[red] (d) edge [loop below] (d);
\path[blue] (b) edge (c);
\path[blue] (d) edge (f);
\path[blue] (a) edge [loop above] (a);
\path[blue] (e) edge [loop below] (e);
\end{tikzpicture}
\hspace{.2in}
\begin{tikzpicture}[node distance=.25cm,>=stealth',bend angle=45,auto]
\tikzstyle{state}=[circle,draw=black,minimum size=4mm]
\node[state] (a) {$1$};
\node[state] (b) [below right=of a,xshift=3mm] {$2$};
\node[state] (c) [below=of b] {$3$};
\node[state] (d) [below left=of c,xshift=-3mm] {$4$};
\node[state] (e) [above left=of d,xshift=-3mm] {$5$};
\node[state] (f) [above=of e] {$6$};
\path[red] (a) edge (c);
\path[red] (f) edge (e);
\path[red] (b) edge [loop above] (b);
\path[red] (d) edge [loop below] (d);
\path[blue] (b) edge (a);
\path[blue] (d) edge (f);
\path[blue] (c) edge [loop below] (c);
\path[blue] (e) edge [loop below] (e);
\end{tikzpicture}
\hspace{.2in}
\begin{tikzpicture}[node distance=.25cm,>=stealth',bend angle=45,auto]
\tikzstyle{state}=[circle,draw=black,minimum size=4mm]
\node[state] (a) {$1$};
\node[state] (b) [below right=of a,xshift=3mm] {$2$};
\node[state] (c) [below=of b] {$3$};
\node[state] (d) [below left=of c,xshift=-3mm] {$4$};
\node[state] (e) [above left=of d,xshift=-3mm] {$5$};
\node[state] (f) [above=of e] {$6$};
\path[red] (b) edge (c);
\path[red] (f) edge (e);
\path[red] (a) edge [loop above] (a);
\path[red] (d) edge [loop below] (d);
\path[blue] (c) edge (a);
\path[blue] (d) edge (f);
\path[blue] (b) edge [loop above] (b);
\path[blue] (e) edge [loop below] (e);
\end{tikzpicture}
\hspace{.2in}
\begin{tikzpicture}[node distance=.25cm,>=stealth',bend angle=45,auto]
\tikzstyle{state}=[circle,draw=black,minimum size=4mm]
\node[state] (a) {$1$};
\node[state] (b) [below right=of a,xshift=3mm] {$2$};
\node[state] (c) [below=of b] {$3$};
\node[state] (d) [below left=of c,xshift=-3mm] {$4$};
\node[state] (e) [above left=of d,xshift=-3mm] {$5$};
\node[state] (f) [above=of e] {$6$};
\path[blue] (a) edge (f);
\path[blue] (b) edge (e);
\path[blue] (c) edge (d);
\path[red] (b) edge (f);
\path[red] (c) edge (e);
\path[red] (a) edge (d);
\path[white] (a) edge [loop above] (a); 
\path[white] (d) edge [loop below] (d); 
\end{tikzpicture}
\end{center}
\caption{The involution product graphs depicting the twelve ways to write the permutation $(123)(456)$ as the product of two involutions, $\sigma\tau$, where $\sigma$ is denoted in red and $\tau$ in blue.}\label{fig:(123)(456) decompositions}
\end{figure}

 Suppose $\rho' = g\rho g^{-1}$ for some $g \in S_n$. Let $\rho = \sigma\tau$ for some involutions $\sigma$ and $\tau$. Then $\sigma' = g\sigma g^{-1}$ and $\tau' = g\tau g^{-1}$ are also involutions, and we have $\rho' = \sigma'\tau'$. Hence, $N(\rho) = N(\rho')$, and we see that the number of ways to write an element of $S_n$ as a product of involutions depends only on its conjugacy class (we could make the same observation in any finite group).  Partitioning $S_n$ by conjugacy class is equivalent to partitioning it by cycle type.  Thus, for any partition $\lambda \vdash n$, we define $N(\lambda)$ to be the number of ways any particular permutation with cycle type $\lambda$ can be written as a product of involutions.  Carter \cite[Theorem C]{Carter} shows that any element in a finite Weyl group can be written as a product of two involutions, so one could study the same question for conjugacy classes of Weyl groups.

Define 
\begin{equation}\label{eq:R}
R_m(k) = \sum_{0\leq i \leq m/2} \frac{k^{m-i}}{i!}\binom{m}{\underbrace{2,2,\ldots,2}_i} = \sum_{0\leq i \leq m/2} \frac{k^{m-i} m!}{2^i i! (m-2i)!}. 
\end{equation}
We will prove the following formula for $N(\lambda)$. 

\begin{thm}\label{thm:formula}
Let $\lambda = 1^{j_1} 2^{j_2} \cdots $. Then 
\begin{equation}\label{eq:N}
 N(\lambda) = \prod_{i=1}^n R_{j_i}(i).
\end{equation}
\end{thm}

Before proving the theorem, we need some preliminary results.

\begin{lem}\label{lem:cyc}
Suppose an involution product graph for $\rho$ has a blue-red path of the form: \[ 1 {\color{blue} \to } v_1 {\color{red} \to} 2 {\color{blue} \to } v_2 {\color{red} \to } 3 {\color{blue} \to } v_3 {\color{red} \to } \cdots {\color{red} \to } k {\color{blue} \to } v_k {\color{red} \to } 1.\] Then both \[ (12\cdots k) \quad \mbox{ and } \quad (v_1\,v_k \cdots v_2)\] are (possibly identical) $k$-cycles in $\rho$.
\end{lem}

\begin{proof}
Fix an involution product graph of $\rho$ with the blue-red path described. Then obviously $(12\cdots k)$ is a $k$-cycle in $\rho$. Consider now $(v_1\, v_k \cdots v_2)$.

Since the involution graph has precisely one edge of each color at every vertex, we see that $v_i \neq v_j$ if $i\neq j$.  That is, the path cannot have both $i {\color{blue} \to } v_i {\color{red} \to } (i+1)$ and $j {\color{blue} \to } v_i {\color{red} \to } (j+1)$ if $i \neq j$. Hence, $|\{v_1,\ldots, v_k\}|=k$, and because all arrows in the product graph are reversible, we get another blue-red path that shows $(v_1 v_k v_{k-1} \cdots v_2)$ is a $k$-cycle in $\rho$:  \[ v_1 {\color{blue} \to } 1 {\color{red} \to} v_k {\color{blue} \to } k {\color{red} \to } v_{k-1} {\color{blue} \to } (k-1) {\color{red} \to } \cdots {\color{red} \to } v_2 {\color{blue} \to } 2 {\color{red} \to } v_1.\] 
\end{proof}

If differently colored cycles in an involution product graph form blue-red path as described in Lemma~\ref{lem:cyc}, then those cycles are \emph{interlaced}.  If a cycle does not interlace with any other, then that cycle is \emph{isolated}.

A first consequence of Lemma \ref{lem:cyc} is the following.

\begin{cor}\label{cor:match}
The connected components of an involution product graph of a permutation $\rho$ each either describe a cycle in $\rho$ or two disjoint cycles of the same size. In particular, 
\begin{enumerate}
\item three or more cycles cannot be interlaced in an involution product graph, and
\item if $k \neq k'$, then a $k$-cycle and a $k'$-cycle cannot be interlaced in an involution product graph.
\end{enumerate}
\end{cor}

The next consequence is the key to our enumeration of involution product graphs

\begin{cor}\label{cor:num}
\begin{enumerate}
\item There are precisely $k$ ways to write a $k$-cycle as a product of involutions.
\item There are precisely $k$ ways to write two disjoint $k$-cycles as a product of involutions, given that their elements are in the same connected component of the involution graph. 
\end{enumerate}
\end{cor}

\begin{proof}
If the cycles from Lemma \ref{lem:cyc} are fixed, then the designation of the label $v_1$ uniquely determines all the edges of this connected component of the involution graph.
\end{proof}

\begin{prop}\label{prp:rec}
Suppose $\lambda = k^m = (k,k,\ldots,k) \vdash km$. Then, \[ N(k^m) = R_m(k), \] where $R_m(k)$ is defined in Equation \eqref{eq:R}.
\end{prop}

\begin{proof}
Let $\rho$ be a fixed permutation of cycle type $\lambda$. By Corollary \ref{cor:match}, an involution graph for $\rho$ either has all the $k$-cycles in disjoint connected components, or some pairs of the cycles are connected. For example, with $m=7$ $k$-cycles, perhaps two pairs of them are interlaced, while the other three are isolated, as in the following picture.
$$
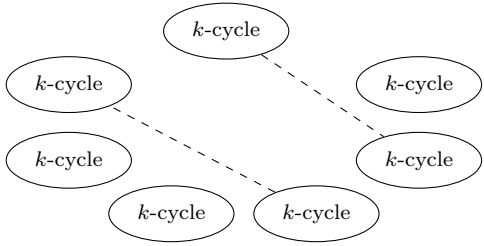
\begin{tikzpicture}[node distance=.25cm,>=stealth',bend angle=45,auto]
\tikzstyle{state}=[ellipse,draw=black,minimum size=4mm]
\node[state] (a) {\tiny $k$-cycle};
\node[state] (b) [below right=of a,xshift=12mm] {\tiny $k$-cycle};
\node[state] (c) [below=of b] {\tiny $k$-cycle};
\node[state] (d) [below left=of c] {\tiny $k$-cycle};
\node[state] (e) [left=of d] {\tiny $k$-cycle};
\node[state] (f) [above left=of e] {\tiny $k$-cycle};
\node[state] (g) [above=of f] {\tiny $k$-cycle};
\path[dashed] (a) edge (c);
\path[dashed] (d) edge (g);
\end{tikzpicture}$$

There are precisely \[\frac{1}{i!}\binom{ m}{ \underbrace{2,2,\ldots,2}_i } = \frac{m!}{2^i i!(m-2i)!} \] ways to choose a matching of $i$ pairs of the $k$-cycles, leaving $m-2i$ of the $k$-cycles isolated. Given such a matching, Corollary \ref{cor:num} tells us there are $k$ ways to draw each connected component, giving  $k^i\cdot k^{m-2i} = k^{m-i}$ choices. Hence, the number of involution product graphs for $\rho$ that have $i$ pairs of $k$-cycles is: \[ \frac{k^{m-i}m!}{2^i i!(m-2i)!}.\] 

To count all possible involution graphs for $\rho$, we sum over all $i$ to obtain the desired result: \[ N(\rho) = \sum_{0\leq i \leq m/2} \frac{k^{m-i} m!}{2^i i! (m-2i)!} = R_m(k).\] 
\end{proof}

We are now able to prove Theorem \ref{thm:formula}.

\begin{proof}[Proof of Theorem \ref{thm:formula}]
Let $\rho$ be a permutation of type $\lambda = 1^{j_1} 2^{j_2}\cdots$. By Corollary~\ref{cor:match}, two cycles of $\rho$ are in different connected components if they are of different size. Hence, for $k\neq k'$, the $k$-cycles and $k'$-cycles are mutually independent, and to count the number of involution product graphs for $\rho$, we simply multiply the number of possibilities for each cycle size:
\[ N(\lambda) = N(1^{j_1})\cdot N(2^{j_2}) \cdots = \prod_{j=1}^n R_{j_i}(i).\]
\end{proof}

\section{Connections with the character table of $S_n$}\label{sec:connections}

Now that we have enumerated the number of pairs of involutions with a fixed product, we will  describe why understanding such involution products may be related to characters of the symmetric group. Our tone is primarily expository here, built upon substantial work of White \cite{white} and Stanton and White \cite{stanton-white}. Omitted details may be found in those papers.

To begin, we recall that the well-known Murnaghan-Nakayama rule provides the following  formula for an irreducible character of the symmetric group $S_n$: \[ \chi_\mu^\lambda = \sum_{T} (-1)^{\h(T)},\] where the sum is taken over all \emph{rim-hook tableaux} $T$ of shape $\lambda$ and content $\mu$, and $\h(T)$ is the sum of the heights of the hooks in $T$ minus the number of hooks. In general, $(-1)^{\h(T)}$ gives the \emph{sign} of a tableau $T$.  For example, \[T=\young(11113,12333,2246,244,5)\] is a rim-hook tableau with content $\mu =(5,4,4,3,1,1)$, its height is calculated by $\h(T) = 2+3+2+2+1+1 - 6 = 5$, and $(-1)^{\h(T)}<0$ so $T$ is a negative tableau. See either \cite[Sec. 4.10]{sagan} or \cite[Sec. 7.17]{ec2} for precise definitions and further discussion.

This formula is lovely in that it gives an explicit combinatorial description for the characters. However, it involves tremendous cancellation. For instance, consider the character table for $S_6$ in Figure \ref{fig:S6}. There are 40 rim-hook tableau of content $(2,2,1,1)$, yet in the column indexed by $\mu = (2,2,1,1)$ we see that only 12 terms are non-cancelling; that is, the sum of the absolute values of the entries in this column is $1 + 1 + 1 + 2 + 1 + 2 + 1 + 1 + 1 + 1 = 12$.

\begin{figure}[h]
\begin{tabular}{l | c c c c c c c c c c c }
$\lambda\backslash\mu$ & 6 & 51 & 42 & 411 & 33 & 321 & 3111 & 222 & 2211 & 21111 & 111111 \\
\hline 
6 & 1 & 1 & 1 & 1 & 1& 1& 1& 1& 1& 1& 1 \\
51 & -1 & 0 & -1 & 1 & -1 & 0 & 2 & -1 &1 & 3 & 5 \\
42 & 0 & -1 &1 &-1 &0 &0 &0 &3& 1& 3& 9 \\
411 & 1 &0&0&0 &1& -1& 1 &-2& -2& 2& 10\\
33 & 0&0& -1& -1& 2& 1& -1& -3& 1& 1& 5\\
321 & 0 &1& 0&0 &-2& 0& -2& 0&0&0& 16 \\
3111 & -1 &0&0&0& 1&1&1& 2& -2& -2 &10 \\
222 & 0 &0 &-1& 1& 2& -1& -1& 3 &1 &-1& 5\\ 
2211 & 0 &-1 & 1& 1& 0&0&0& -3& 1& -3& 9\\
21111 & 1 &0 &-1& -1& -1& 0 &2 &1& 1& -3& 5 \\
111111 & -1 &1 &1& -1& 1& -1&1& -1& 1& -1& 1\\
\hline
$N(\mu)$ & 6 & 5 & 8 & 8 & 12 & 6 & 12 & 20 & 12 & 20 & 76
\end{tabular}
\caption{The character table of $S_6$, along with $N(\mu)$, the number of ways to write a permutation of cycle type $\mu$ as a product of involutions.}\label{fig:S6}
\end{figure}

On the other hand, there are precisely 12 rim hook tableaux of content (3,3), and so there is no cancellation in this column. As we will see shortly, it follows from work of White \cite{white} that, when all parts of $\mu$ are the same size, the Murnaghan-Nakayama formula for $\chi_\mu^\lambda$ is \emph{cancellation-free} (Corollary \ref{cor:MNcancelfree}). Such a partition $\mu$ is called \emph{rectangular}.  Our feeling is that understanding involution products may lead to a similar cancellation-free formula for any content $\mu$. 

The idea begins with the following observation. For any $\lambda$, the Murnaghan-Nakayama rule gives \[\chi^{\lambda}_{1^n} = |\{ \mbox{standard Young tableaux of shape } \lambda\}|.\] On the other hand, the Schensted insertion algorithm  gives a bijection between the set of standard Young tableaux with $n$ boxes and the set of involutions in $S_n$. If we let $I^{\lambda} = \{ \sigma \in S_n : \sigma^2 = 1, \sh(\sigma) = \lambda\}$, where $\sh(\sigma)$ denotes the shape of the image of $\sigma$ upon the Schensted insertion, then \[ \chi^\lambda_{1^n} = |I^\lambda|.\] To generalize this idea to $\mu \neq 1^n$, we first need to discuss a more general insertion algorithm. 

In \cite{white}, White defines a map \sch\ that generalizes the usual Schensted insertion algorithm. Given a pair of rim hook tableaux $(P,Q)$ of the same shape and content, we say the pair is \emph{positive} if $P$ and $Q$ have the same sign; the pair is \emph{negative} otherwise. The map \sch\ provides a bijection that maps a positive pair $(P,Q)$ of content $\mu$ to either a negative pair $(P',Q')$ of tableaux of content $\mu$ or to a ``hook permutation'' of content $\mu$. (It will not be necessary to give a precise definition of \sch\ here. The interested reader can refer to \cite{white}.)  The case of rectangular content is particularly nice. 

\begin{prop}[\cite{white}, Corollaries 9 and 10]\label{prop:no cancellation for rectangular content}
If $\mu = k^m$, then \sch\ is a bijection from hook permutations of content $\mu$ to the set of all pairs $(P,Q)$ of rim hook tableaux of the same shape and content $\mu$. In particular, every pair of tableaux $(P,Q)$ is positive.
\end{prop}

The following is immediate.

\begin{cor}\label{cor:MNcancelfree}
Any two rim hook tableaux of the same shape and the same rectangular content have the sign.  Hence, \[ |\chi^{\lambda}_{k^m}| = |\{ P : P \mbox{ has shape $\lambda$ and content $k^m$} \}|.\] In other words, there is no cancellation in the Murnaghan-Nakayama formula for $\chi^{\lambda}_\mu$ when $\mu$ is rectangular.
\end{cor}

We will now see how this result is related to involution products.

To any pair of involutions we may assign a special kind of hook permutation, which we call a ``hook-block involution," as follows.  We create this involution by looking at all of the the $k$-cycles in the involution product graph, in decreasing order of $k$.  Suppose that we have examined the $k'$-cycles for all $k'>k$, and so far created a hook-block involution on the letters $\{1,\ldots,M\}$.  Now suppose that there are $r$ $k$-cycles in the involution product graph.  Label these $\{M+1, \ldots, M+r\}$ in order of the smallest element appearing in each.  We thus inherit an involution of $\{M+1, \ldots, M+r\}$, based on which $k$-cycles are interlaced.  To each letter in this involution, we associate a hook, where the choice of $v_1$ in Lemma \ref{lem:cyc} determines the height of the first column in each hook of the 2-cycle: if $v_1$ is in the $i$th position when the cycle is written with the smallest letter in the first position, then the height of that first column in the hook is $i$.  Note that any pair of letters that are each other's images under the involution get assigned the same hook shape, by Lemma~\ref{lem:cyc}.  We demonstrate with an example.

\begin{example}
\[ \begin{minipage}{1.5in}
\begin{tikzpicture}[node distance=.25cm,>=stealth',bend angle=45,auto]
\tikzstyle{state}=[circle,draw=black,minimum size=4mm]
\node[state] (a) {$1$};
\node[state] (b) [below right=of a,xshift=5mm] {$2$};
\node[state] (c) [below=of b] {$3$};
\node[state] (d) [below left=of c] {$4$};
\node[state] (e) [left=of d] {$5$};
\node[state] (f) [above left=of e] {$6$};
\node[state] (g) [above=of f] {$7$};
\path[red] (a) edge (d);
\path[red] (b) edge (c);
\path[red] (e) edge[bend right] (g);
\path[red] (f) edge [loop left] (f);
\path[blue] (a) edge [loop above] (a);
\path[blue] (b) edge (g);
\path[blue] (c) edge[bend right] (e);
\path[blue] (d) edge[bend right] (f);
\end{tikzpicture}
\end{minipage}
\longleftrightarrow \left(\, \young(111) \quad \young(3,3) \quad \young(2,2) \, \right)\]
First of all, the product is $(146)(25)(37)$, which has cycle type $(3,2,2)$.  Hence we know the hook involution will have a 3-hook and two 2-hooks. Since 1 is the smallest element of the 3-cycle, we see where its blue arrow points to determine the height of the 3-hook. It points to itself, so we choose the hook of height 1. (If the 1 had had a blue arrow to the 4, then the 3-hook would have had height 2. If it had pointed to the 6 it would have had height 3.) There are two 2-cycles, and they lie in the same connected component of the matching graph. Hence, the corresponding 2-hooks are transposed in the hook permutation. The shape of the hooks is determined by where the 2 is connected with its blue edge. If it had been connected to the 3, then the 2-hooks would be horizontal. However, the blue edge connects to the 7 (which is the second-largest element in its cycle), and so the hooks have height 2.
\end{example}

In the case where all the hooks are the same size, we get what Stanton and White \cite{stanton-white} call a ``$k$-partial involution." Indeed many things work nicely when the cycles all have the same size. Stanton and White \cite{stanton-white} show the following.

\begin{thm}[\cite{stanton-white}]\label{thm:k-bijection}
A hook permutation $\mathcal{H}$ of type $\mu = k^m$ is a $k$-partial involution if and only if $\sch(\mathcal{H}) = (P,P)$. That is, we have a bijection,
\[ \{ \mbox{ $k$-partial involutions } \} \xleftrightarrow{\sch} \{ \mbox{ rim-hook tableaux $P$ of content $k^m$ } \}.\]
\end{thm}

Now let \[I_k^{\lambda} = \{\mathcal{H} : \mathcal{H} \text{ is a } k\text{-partial involution and } \sch(\mathcal{H}) \text{ has shape } \lambda\}.\] Then, together with Corollary \ref{cor:MNcancelfree}, we obtain the following.
\begin{cor}\label{cor:rectangular equality}
Let $k$ and $m$ be positive integers. Then for all $\lambda \vdash n$, \[ |\chi^{\lambda}_{k^m}| = |I_k^{\lambda}|.\]
\end{cor}

Because $N(\mu)$ counts all hook-block involutions of content $\mu$, Theorem \ref{thm:formula} implies the following result.
\begin{cor}\label{cor:rectangular sum}
Let $k$ and $m$ be positive integers. Then, \[ \sum_{\lambda \vdash km} |\chi^\lambda_{k^m}| = N(k^m) = \sum_{0\leq i \leq m/2} \frac{k^{m-i} m!}{2^i i! (m-2i)!}.\] 
\end{cor}
This leads one to wonder whether something similar is true for non-rectangular $\mu$.

\begin{ques}\label{ques:1}
For which $\mu \vdash n$ is it true that
\begin{equation}\label{eq:q1}
 \sum_{\lambda \vdash n} |\chi^\lambda_\mu| = N(\mu)?
\end{equation}
\end{ques}

Corollary \ref{cor:rectangular equality} shows that Equation \eqref{eq:q1} holds for $\mu =k^m$  and by computer we have checked that it holds for all $\mu\vdash n$ when $n \leq 7$. However, when $n=8$, we begin to see some discrepancies. In each of these cases $N(\mu)$ is greater than $\sum_{\lambda \vdash n} |\chi^\lambda_\mu|$. See Table \ref{tab:diff}, where for $n\leq 10$ we have listed all the shapes $\mu\vdash n$ for which Equation~\eqref{eq:q1} does not hold.

\begin{table}
\subtable{
\begin{tabular}{c | c}
$\mu$ & $N(\mu) - \sum_{\lambda \vdash n} |\chi^\lambda_\mu|$\\
\hline
\hline
$2^2 4$ & 4\\
$1^2 2 4$ & 4 \\
$1^4 4$ & 4\\
$1^2 2^3$ & 4 \\
$1^4 2^2$ & 8\\
$1^6 2$ & 4\\
\hline
$234$ & 4 \\
$1^234$ & 4\\
$1 2^2 4$ & 4\\
$1^3 24$ & 4\\
$1^5 4$ & 4\\
$1^3 3^2$ & 8\\
$2^3 3$ & 4\\
$1^22^2 3$ & 4\\
$1^4 23$ & 4\\
$1^6 3$ & 20\\
$1^3 2^3$ & 4\\
$1^5 2^2$ & 8\\
$1^7 2$ & 4\\ 
\end{tabular}
}
\hspace{.5cm}
\subtable{
\begin{tabular}{c | c }
$\mu$ & $N(\mu) - \sum_{\lambda \vdash n} |\chi^\lambda_\mu|$\\
\hline
\hline
$235$ & 4\\
$1^2 3 5$ & 4 \\
$1^3 25$ & 4\\
$1^5 5$ & 8 \\
$2^3 4$ & 8\\
$1^2 2^2 4$ & 8\\
$1^424$ & 8 \\
$1^64$ & 8\\
$2^2 3^2$ & 8\\
$1^4 3^2$ & 8\\
$1^3 2^2 3$ & 8\\
$1^5 23$ & 16\\
$1^7 3$ & 56\\
$1^22^4$ & 8\\
$1^4 2^3$ & 28\\
$1^6 2^2$ & 48\\
$1^8 2$ & 32\\
\end{tabular}
}
\caption{When Equation \eqref{eq:q1} does not hold.}\label{tab:diff}
\end{table}

If we check all shapes $\mu \vdash n$ with $n\leq 16$, we find that there are indeed a number of cases in which Equation~\eqref{eq:q1} holds. We summarize the number of agreements ($N(\mu) = \sum_{\lambda \vdash n} |\chi^\lambda_\mu|$) and discrepancies ($N(\mu) > \sum_{\lambda \vdash n} |\chi^\lambda_\mu|$) in Table \ref{tab:numdiff}. We have also included the number of rectangular partitions $\mu = k^m$ since in these cases we know that equality holds. For example, when $n=16$ there are 64 shapes $\mu \vdash 16$ for which $N(\mu) = \sum_{\lambda \vdash n} |\chi^\lambda_\mu|$, yet we can only explain 5 of these: $1^{16}, 2^8, 4^4, 8^2,$ and $16^1$.

\begin{table}
\begin{tabular}{ c | c c c c c c c c c c} 
$n$ & $\leq 7$ &8 & 9 & 10 & 11 & 12 & 13 & 14 & 15 & 16\\
\hline
\hline
Discrepancies: & 0 & 6 & 13 & 17 & 24 & 46 & 60 & 83 & 114 & 167\\
Agreements: & all &16 & 17 & 25 & 32 & 31 & 41 & 52 & 62 & 64\\
Known agreements $(\mu=k^m)$: & - & 4 & 3 & 4 & 2 & 6 & 2 & 4 & 4 & 5\\
\hline
Partitions of $n$: & - & 22 & 30 & 42 & 56 & 77 & 101 & 135 & 176 & 231
\end{tabular}
\vspace{.5cm}
\caption{Counting the number of shapes $\mu$ for which $N(\mu) = \sum_{\lambda \vdash n} |\chi^\lambda_\mu|$.}\label{tab:numdiff}
\end{table}

If we are to move beyond the rectangular case, we need to understand how to associate involution products, or hook-block involutions, to a collection of rim-hook tableaux.  Theorem \ref{thm:k-bijection} does this in the case when $\mu$ is rectangular, and Theorem~\ref{thm:good involutions} shows that \sch\ does something similar in the general case. See Figure \ref{fig:schematic} for an illustration of these results and their scopes.

\begin{thm}\label{thm:good involutions}
A hook permutation $\mathcal{H}$ of content $\mu$ is a hook-block involution if and only if $\sch(\mathcal{H}) = (P,P)$.  That is, the non-cancelling pairs $(P,P)$ of content $\mu$ are in bijection with the set of hook-block involutions.
\end{thm}

\begin{figure}
\begin{tikzpicture}
\draw[very thick, rounded corners] (0,0) rectangle (5,7);
\draw (2.5,5) node {Positive pairs $(P,Q)$};
\draw[very thick, rounded corners] (10,0) rectangle (15,7);
\draw (12.5,6) node {Negative pairs $(P',Q')$};
\draw[thick] (10,5)--(15,5);
\draw (12.5,4) node {Hook permutations};
\draw[thick, rounded corners] (1,.5) rectangle (4,3);
\draw (2.5,2.25) node {Non-cancelling};
\draw (2.5,1.75) node {Diagonal pairs};
\draw (2.5,1.25) node {$(P,P)$};
\draw[thick, rounded corners] (11,.5) rectangle (14,3);
\draw (12.5,2) node {Hook-block};
\draw (12.5,1.5) node {Involutions};
\draw[thick, <->] (5.1,5) -- node[above] {\sch} (9.9,5);
\fill[white] (4.5,1.9) rectangle (10.5,2.1);
\draw[thick, <->] (4.1,2) -- node[above] {Restriction of \sch} node[below] {(Theorem \ref{thm:good involutions})} (10.9,2);
\end{tikzpicture}
\caption{Schematic diagram for the \sch\ bijection with fixed shape $\lambda$ and content $\mu$.}\label{fig:schematic}
\end{figure}
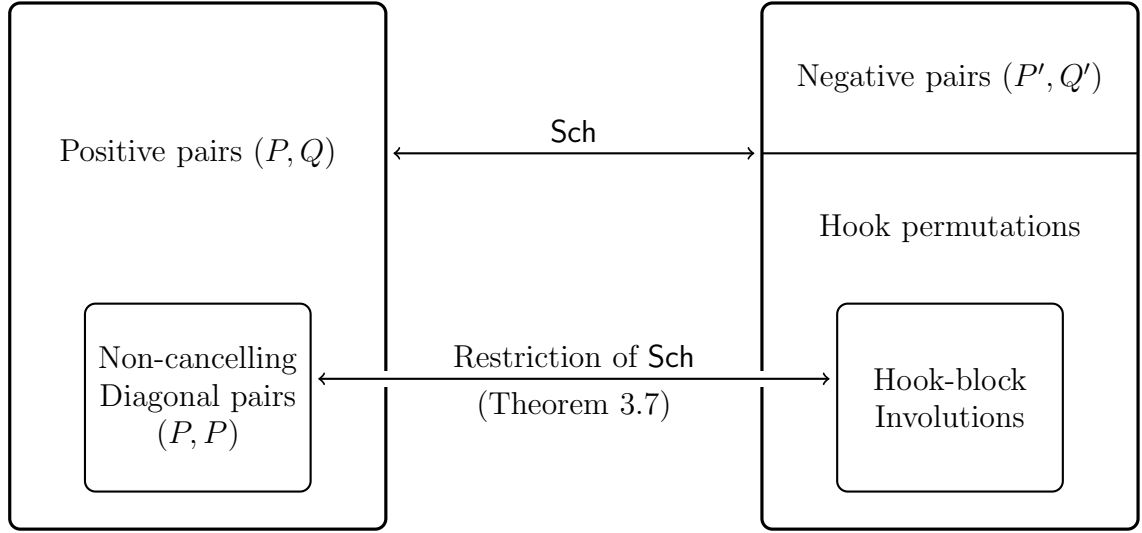

We outline the proof of this theorem, making references to the appropriate work of Stanton and White.

\begin{proof}[Proof of Theorem~\ref{thm:good involutions}]
The result is proved by induction on the number of different part sizes in the content $\mu$.  When there is only one part size, the content is rectangular, and this case is covered by Stanton and White.

Suppose, inductively, that the shape $\alpha$ has been constructed by inserting all hooks of length greater than $\ell$.  We observe that when the insertion algorithm \sch\ is inserting hooks of length $\ell$, the process is unaffected by appending $\ell$ sufficiently long rows, each of length at least $|\alpha| + \ell j_{\ell}$, on top of $\alpha$.  We now do this, and further require that the size of the resulting shape $\beta$ is a multiple of $\ell$.

The positioning of these $\ell$-hooks into the shape $\beta$ does not affect the content of the previously constructed shape, because the content of each of these $\ell$-hooks is greater than any content appearing in $\beta$.  Thus, we can ignore the labels in the already constructed shape.  Moreover, if we pretend that $\beta$ had been filled entirely by $\ell$-hooks of smaller content than any of the $\ell$-hooks we are about to insert, we see by Stanton and White's result that all of these $\ell$-hooks, both those we are ``imagining'' and those we are inserting next, must form a $\ell$-partial involution.  And finally, since those ``imagined'' $\ell$-hooks already inserted into $\beta$ must be an $\ell$-partial involution amongst themselves, these new $\ell$-hooks that we are about to insert must form an $\ell$-partial involution amongst themselves as well.
\end{proof}

Theorem \ref{thm:good involutions} shows it is possible to associate a collection of $N(\mu)$ rim-hook tableaux of content $\mu$ to the set of all hook-block involutions of content $\mu$. Unfortunately, it is not obvious how to deduce when $\sum_{\lambda\vdash n} |\chi^{\lambda}_{\mu}|$ and $N(\mu)$ coincide from this result. Moreover, as the following example shows, even when $N(\mu)$ and $\sum_{\lambda\vdash n} |\chi^{\lambda}_{\mu}|$ are equal, applying \sch\ to the set of hook-block involutions does not necessarily yield a result like Corollary \ref{cor:rectangular equality}. 

\begin{example}\label{ex:nonrec} While $N(2211) = 12$, applying \sch\ to the corresponding twelve hook-block involutions yields no tableaux of shape 6 or shape 51, while $\chi^6_{2211}= \chi_{2211}^{51} = 1$. Compare Figures \ref{fig:hooks2} and \ref{fig:tabl2} with the column corresponding to 2211 in Figure \ref{fig:S6}.
\end{example}

\begin{figure}[h]
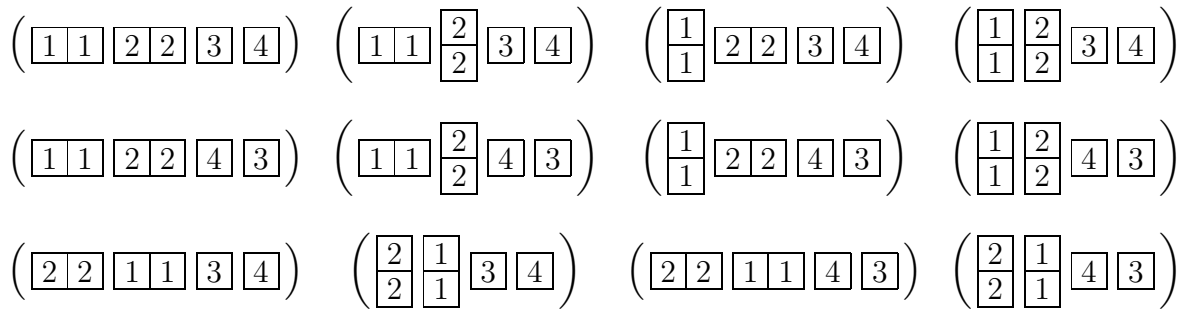

\[\begin{array}{c c c c}
\Big( \, \young(11) \,\, \young(22) \,\, \young(3) \,\, \young(4) \, \Big) & \left( \, \young(11) \,\, \young(2,2) \,\, \young(3) \,\, \young(4) \, \right) & \left(\, \young(1,1) \,\, \young(22) \,\, \young(3) \,\, \young(4) \, \right) & \left( \, \young(1,1) \,\, \young(2,2) \,\, \young(3) \,\, \young(4) \, \right) \\
&&&\\
\Big( \, \young(11) \,\, \young(22) \,\, \young(4) \,\, \young(3) \, \Big) & \left( \, \young(11) \,\, \young(2,2) \,\, \young(4) \,\, \young(3) \, \right) & \left(\, \young(1,1) \,\, \young(22) \,\, \young(4) \,\, \young(3) \,\right) & \left(\, \young(1,1) \,\, \young(2,2) \,\, \young(4) \,\, \young(3) \,\right) \\ &&&\\
\Big( \, \young(22) \,\, \young(11) \,\, \young(3) \,\, \young(4) \,\Big) & \left( \, \young(2,2) \,\, \young(1,1) \,\, \young(3) \,\, \young(4) \,\right) & \Big(\, \young(22) \,\, \young(11) \,\, \young(4) \,\, \young(3) \,\Big) & \left(\, \young(2,2) \,\, \young(1,1) \,\, \young(4) \,\, \young(3) \,\right)
\end{array}\]
\caption{The hook-block involutions corresponding to the involution product graphs for $\sigma = (12)(34)(5)(6)$.}\label{fig:hooks2}
\end{figure}

\begin{figure}[h]
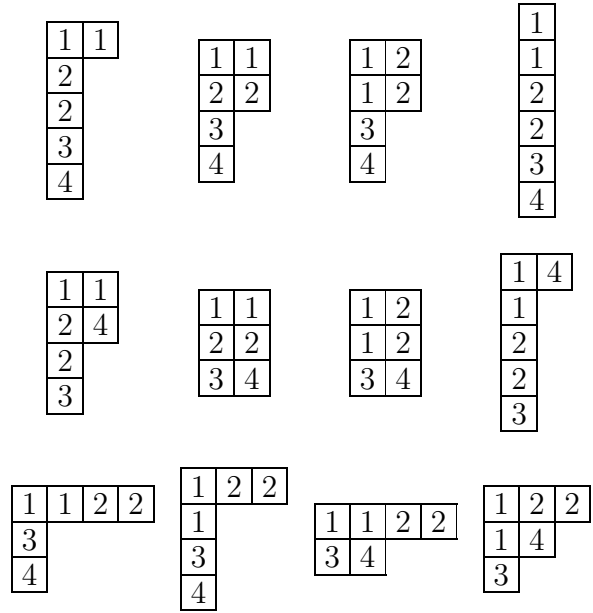

\[\begin{array}{ c c c c }
\young(11,2,2,3,4) & \young(11,22,3,4) & \young(12,12,3,4) & \young(1,1,2,2,3,4)\\
&&&\\
\young(11,24,2,3) & \young(11,22,34) & \young(12,12,34) & \young(14,1,2,2,3) \\
&&&\\
\young(1122,3,4) & \young(122,1,3,4) & \young(1122,34) & \young(122,14,3)
\end{array}
\]
\caption{The image under \sch\ of the hook-block involutions of content $(2,2,1,1)$.}\label{fig:tabl2}
\end{figure}

\section*{Acknowledgments}
The authors wish to thank Allan Berele, Ira Gessel, and John Stembridge for conversations that helped prompt us to better understand characters of $S_n$. Thanks to Dennis White for helping us understand \sch, and thanks to Peter Winkler for helpful discussions related to the formula for $N(\mu)$. Lastly, we should mention the role played by OEIS, without which we would never have guessed at a connection between involution products and characters of the symmetric group.

\end{document}